\documentclass[12pt]{amsart}

\usepackage[english]{babel}
\usepackage[T1]{fontenc}
\usepackage[latin1]{inputenc}
\usepackage{lmodern}
\usepackage{a4}
\usepackage[T1]{fontenc}
\usepackage[all]{xy}
\usepackage{verbatim}
\usepackage{stmaryrd}
\usepackage{amsfonts}
\usepackage{amsmath,amssymb,amsthm}
\usepackage{enumerate}
\usepackage{xspace}
\usepackage{ulem}
\usepackage{euscript}
\usepackage{epsfig}
\usepackage{fancyhdr}
\usepackage{txfonts}
\usepackage{babel,indentfirst}

\makeatletter \@addtoreset{equation}{section}

\usepackage[all]{xy}

\theoremstyle{definition}

\newtheorem{thm}{Theorem}[section]

\newtheorem{cor}{Corollary}[section]

\newcommand{\B}{\mathcal{B}}

\newcommand{\N}{\mathbb{N}}

\title{ New formulas for the  spectral radius  via Aluthge transform}

\author{Fadil Chabbabi and Mostafa Mbekhta}
\address{Universit\'e  Lille1, UFR de Math\'ematiques, Laboratoire CNRS-UMR 8524 P. Painlev\'e, 
59655 Villeneuve d'Ascq  Cedex, France}

\email{Fadil.Chabbabi@Math.univ-lille1.fr}

\email{Mostafa.Mbekhta@math.univ-lille1.fr}


\keywords{ spectral radius, polar decomposition, $\lambda$-Aluthge transform, normaloid operator\\
This work was supported in part by the Labex CEMPI  (ANR-11-LABX-0007-01)}

\date{}

\begin{document}
\maketitle

\begin{abstract}
 
 In this paper we give several expressions of spectral radius of a bounded operator on a Hilbert space, 
  in terms of iterates of  Aluthge transformation, 
 numerical radius and the asymptotic behavior of the powers of this operator.
 
Also we obtain  several  characterizations of normaloid operators.
\end{abstract}

\section{Introduction}
Let $H$  be  complex Hilbert spaces  and  $\mathcal{B}(H)$ be the Banach space of all bounded linear operators from $H$ into it self.

 For $T\in \B(H)$, the spectrum of $T$ is denoted by $\sigma(T)$ and $r(T)$ its spectral radius. We denote also by $W(T)$ and $w(T)$ the numerical range and the numerical radius of $T$.

As usually, for $T\in \B(H)$ we denote the module of $T$ by $|T|=(T^*T)^{1/2}$ and we shall always write, without further
mention, $T=U|T|$ to be the unique polar decomposition  of $T$, where $U$ is the appropriate partial isometry satisfying  $\mathcal{N}(U)=\mathcal{N}(T)$. The Aluthge transform  introduced in \cite{alu}  as 
 $$\Delta(T)=|T|^{\frac{1}{2}}U|T|^{\frac{1}{2}}, \quad  T\in \B(H),$$
   to  extend some properties  of hyponormal  operators.
Later, in \cite{oku},    Okubo introduced a more general notion called  $\lambda-$Aluthge
transform which has also been studied in detail.

 For $\lambda \in [ 0, 1]$, the $\lambda$-Aluthge transform is  defined by, 
 $$
 \Delta_{\lambda}(T)=|T|^{\lambda}U|T|^{1-\lambda}, \quad T\in \B(H).
 $$
  Notice  that $\Delta_0(T) =U|T|=T$, and $\Delta_1(T) = |T|U$ which is known as Duggal's
  transform. It has since been studied in many different contexts and considered by a
  number of authors (see for instance,  \cite{ alu, ay, ams, kp1, kp2,  kp3} and some
  of the references there).  The interest of the Aluthge transform lies in the fact that
   it respects many   properties of the original operator. For example, (see \cite[Theorems 1.3, 1.5]{kp3})
\begin{equation}
   \sigma(\Delta_{\lambda}(T)) = \sigma(T),   \mbox{ for every } \; \; T \in \B(H), 
\end{equation} 
  
 Another important property is that $Lat(T)$, the lattice of $T$-invariant subspaces of
 $H$, is nontrivial if and only if $Lat(\Delta(T))$ is nontrivial
(see \cite[Theorem 1.15]{kp3}). 

Moreover,  Yamazaki (\cite{yam}) (see also, \cite{wan, tam}), established the following interesting  formula for the spectral radius  
   
   \begin{equation}
   \lim_{n\to \infty}\|\Delta_{\lambda}^n(T)\|=r(T)
   \end{equation} 
   where $\Delta_{\lambda}^n$ the $n$-th iterate of $\Delta_{\lambda}$, i.e  $\Delta_{\lambda}^{n+1}(T)=\Delta_{\lambda}(\Delta_{\lambda}^n(T))$, $\Delta_{\lambda}^0(T)=T$.  
  
  \bigskip 
  
In this paper we give several expressions of the spectral radius of an operator. Firstly in terms of the Aluthge transformation (section 2), and secondly, in section 3, we give several expressions of spectral radius,  based on numerical radius and Aluthge transformation. Also, We infer several characterizations of normaloid operators (i.e. $r(T) = \Vert T\Vert$).

\section{ formulas of spectral radius via Aluthge transform}

In this section, we use the  of Rota's Theorem,  order to obtain new formulas of spectral radius via Aluthge transformation

\begin{thm}
 For every operator  $T\in\B(H)$,  we have 
\begin{eqnarray*}
r(T)&=&\inf\{\|\Delta_{\lambda}(XTX^{-1})\|, \: X\in\B(H) \;\;\text{invertible}\;\}
\\&=&\inf\{\|\Delta_{\lambda}(e^A Te^{-A})\|, \: A\in\B(H) \;\;\text{self adjoint}\;\}.
\end{eqnarray*}
\end{thm}
\begin{proof}
For every invertible operator $X\in\B(H)$, by (1.1), we have 
$$\sigma(\Delta_{\lambda}(XTX^{-1}))=\sigma(XTX^{-1})=\sigma(T).$$
It follows that 
$$r(T)=r(\Delta_{\lambda}(XTX^{-1}))\leq \|\Delta_{\lambda}(XTX^{-1})\| \;\; \text{ for every invertible operator}\;\; X\in\B(H).$$
Hence 
 \begin{eqnarray*}
 r(T)&  \leq  & \inf\{ \Vert \Delta_{\lambda}(X T X^{-1})\Vert; \; \;  X \in \B(H) \;\;\text{invertible}\, \}\\
  & \leq &  \inf\{ \Vert \Delta_{\lambda}(\exp(A) T \exp(-A))\Vert; \; \;  A  \in \B(H)  \;\;\text{self adjoint} \,\},
 \end{eqnarray*}
In the other hand, for $\varepsilon >0$, we have 
$$r\big(\dfrac{T}{r(T)+\varepsilon}\big)=\dfrac{r(T)}{r(T)+\varepsilon}< 1.$$
From Rota's Theorem \cite[Theorem 2]{rota},  $\dfrac{T}{r(T)+\varepsilon}$ is similar to an contraction. Thus there exists an invertible operator  $X_\varepsilon \in\B(H)$ such that 
\begin{equation}
\|\Delta_\lambda(X_\varepsilon T X_\varepsilon^{-1})\|\leq \|X_\varepsilon T X_\varepsilon^{-1}\|\leq r(T)+\varepsilon.
\end{equation}
Now, let $X_\varepsilon=U_\varepsilon |X_\varepsilon|$ be the polar decomposition of $X_\varepsilon$. Clearly $U_\varepsilon$ is a unitary operator, and $|X_\varepsilon|$ is invertible. Therefore there exist $\alpha>0$ such that $\sigma(|X_\varepsilon|)\subseteq [\alpha, +\infty[$. Consequently $A_\varepsilon= \ln(|X_\varepsilon|)$ exits and it is self adjoint, we also have 
$$|X_\varepsilon|=e^{A_\varepsilon}\;\;\text{and}\;\; |X_\varepsilon|^{-1}=e^{-A_\varepsilon}.$$
Thus 
\begin{eqnarray*}
\|\Delta_\lambda(X_\varepsilon T X_\varepsilon^{-1})\|&=&\|\Delta_\lambda(Ue^{A_\varepsilon}T e^{-A_\varepsilon} U^*)\|\\
&=& \|U\Delta_\lambda(e^{A_\varepsilon}T e^{-A_\varepsilon} )U^*\|\\
&=&\|\Delta_\lambda(e^{A_\varepsilon}T e^{-A_\varepsilon} )\|.
\end{eqnarray*}
Hence $$\|\Delta_\lambda(e^{A_\varepsilon}T e^{-A_\varepsilon} )\|\leq \|X_\varepsilon T X_\varepsilon^{-1}\|\leq r(T)+\varepsilon.$$
It follows that for all $ \varepsilon >0$,
\begin{eqnarray*}
r(T)&\leq &\inf\{\|\Delta_{\lambda}(XTX^{-1})\|, \:X\in\B(H) \;\;\text{invertible}\;\}\\
&\leq &\inf\{\|\Delta_{\lambda}(e^A Te^{-A})\|, \:A\in\B(H) \;\;\text{self adjoint}\;\}\\
&\leq & \|\Delta_\lambda(e^{A_\varepsilon}T e^{-A_\varepsilon} )\|\leq \|X_\varepsilon T X_\varepsilon^{-1}\|\\
&\leq & r(T)+\varepsilon.
\end{eqnarray*}
Finally, since $ \varepsilon >0$ is arbitrary, we obtain
\begin{eqnarray*}
r(T)&=&\inf\{\|\Delta_{\lambda}(XTX^{-1})\|, \:X\in\B(H) \;\;\text{invertible}\;\}
\\&=&\inf\{\|\Delta_{\lambda}(e^A Te^{-A})\|,  \:A\in\B(H) \;\;\text{self adjoint}\;\}.
\end{eqnarray*}
Theroforte the proof of Theorem is complete.
\end{proof}

As immediate consequence of the Theorem 2.1, we obtain the following corollary
  which gives a formula of the spectral radius based on  $n$-th iterate of $\Delta_{\lambda}$.
\begin{cor} If $T\in\B(H)$, then for every $ n \geq 0$, 
\begin{eqnarray*}
r(T)&=&\inf\{\|\Delta_{\lambda}^n(XTX^{-1})\|, \:X\in\B(H) \;\;\text{invertible}\;\} 
\\&=&\inf\{\|\Delta_{\lambda}^n(e^A Te^{-A})\|, \:A\in\B(H) \;\;\text{self adjoint}\;\}.
\end{eqnarray*}
\end{cor}
\begin{proof}
First, note that $\|\Delta_\lambda(T)\| \leq \|T\|$, consequently we have 
\begin{equation}
\|\Delta_\lambda^n(T)\|\leq \|\Delta_\lambda^{n-1}(T)\|\leq ...\leq \|\Delta_\lambda(T)\|\leq \|T\|, \;\forall n\in\N^*.
\end{equation}
Now, clearly $\sigma(\Delta_\lambda^n(T))=\sigma(T)$, for all $n\in\N$. It follows that, for every invertible operator $X\in \B(H)$ we have
\begin{eqnarray*}
r(T)&=&r(\Delta_\lambda^n(XTX^{-1}))\\&\leq &\|\Delta_\lambda^n(XTX^{-1})\|\\& \leq &\|\Delta_\lambda(XTX^{-1})\|.
\end{eqnarray*}
Therefore 
\begin{eqnarray*}
r(T) &\leq & \inf\{\|\Delta_{\lambda}^n(XTX^{-1})\|,  \: X\in\B(H) \;\;\text{invertible}\;\} \\&\leq & \inf\{\|\Delta_{\lambda}^n(e^A Te^{-A})\|, ~~~A\in\B(H) \;\;\text{self adjoint}\;\}\\
&\leq &\inf\{\|\Delta_{\lambda}(e^A Te^{-A})\|, \:A\in\B(H) \;\;\text{self adjoint}\;\}\\&=&r(T).
\end{eqnarray*}
\end{proof}

An operator $T$ is said to be normaloid if   $r(T) = \Vert T\Vert$.

 As immediate consequence of the Corollary 2.1, we obtain the following corollary
  which is a characterization of normaloid operators  via $\lambda$-Aluthge transformation :
\begin{cor} If $T\in\B(H)$, then the following assertions are equivalent

(i) $T$ is normaloid;

(ii) $\Vert T\Vert \leq  \|\Delta_\lambda(XTX^{-1})\|,  \: \text{for all invertible } X\in \B(H)$;

(iii)  $\Vert T\Vert \leq  \|\Delta_\lambda^n(XTX^{-1})\|,  \: \text{for all invertible } X\in \B(H)  \: \text{ and  for all natural number } n$.

\end{cor}

As immediate consequence of the Corollary 2.2, we obtain a new characterization of normaloid operators

\begin{cor} If $T\in\B(H)$, then the following assertions are equivalent

(i) $T$ is normaloid;

(ii) $\Vert T\Vert \leq  \|XTX^{-1}\|,  \: \text{for all invertible } X\in \B(H)$;

\end{cor}

\begin{thm} Let $T\in \B(H)$. Then for each natural number n, we have
\begin{eqnarray*}
r(T)&=&\lim_k \|\Delta_\lambda^n(T^k)\|^{1/k}\\
&=& \lim_k\|\Delta_\lambda(T^k)\|^{1/k}.
\end{eqnarray*}
\end{thm}
\begin{proof}
Note that,
\begin{equation}
r(T)=r(\Delta_\lambda^n(T))\leq \|\Delta_\lambda^n(T)\|\leq \|\Delta_\lambda(T)\|\leq \|T\| \;\;\forall n\in \N^*.
\end{equation}
Let $k\in \N$ be arbitrary, we have 
 $$r(T)^k=r(T^k)=r(\Delta_\lambda^n(T^k))\leq \|\Delta_\lambda^n(T^k)\|\leq\|\Delta_\lambda(T^k)\|\leq \|T^k\| \;\;\forall n\in \N.$$
Hence 
$$r(T)\leq \| \Delta_\lambda^n(T^k)\|^{1/k}\leq\|\Delta_\lambda(T^k)\|^{1/k}\leq \|T^k\|^{1/k}.$$
Therefore 
$$r(T)\leq \lim_k \| \Delta_\lambda^n(T^k)\|^{1/k}\leq \lim_k\|\Delta_\lambda(T^k)\|^{1/k}\leq \lim_k \|T^k\|^{1/k}=r(T).$$
Which completes the proof.
\end{proof}

As immediate consequence of  Theorem 2.2, we obtain the following corollary
  which is a new  characterization of normaloid operators
\begin{cor} If $T\in\B(H)$, then the following assertions are equivalent

(i) $T$ is normaloid;

(ii) $\Vert T\Vert^k = \|\Delta_\lambda(T^k)\|,  \: \text{for all natural number } k$;

(iii) $\Vert T\Vert^k = \|\Delta_\lambda^n(T^k)\|, \: \text{for every natural number } k, n$
\end{cor}

\section{ Spectral radius via numerical radius and Aluthge transform}
For $T\in \B(T)$, we denote the numerical range and numerical radius  of $T$ by $W(T)$ and $w(T)$, respectively. 
$$ W(T) = \{ <Tx, x>; \; \; \Vert x\Vert = 1\} \quad \text{and} \quad w(T) = \sup\{ \vert \lambda \vert; \; \; \lambda \in W(T)\}.$$

In the following theorem, we obtain a new expression of the spectral radius  by means of the numerical radius and Aluthge transform.
\begin{thm}
 For every operator  $T\in\B(H)$ and for each natural number n,  we have 
\begin{eqnarray*}
r(T)&=&\inf\{ w(\Delta_{\lambda}^n(XTX^{-1})), \quad X\in\B(H) \;\;\text{invertible}\;\}
\\&=&\inf\{w(\Delta_{\lambda}^n(e^A Te^{-A})),  \quad A\in\B(H) \;\;\text{self adjoint}\;\}.
\end{eqnarray*}
\end{thm}
\begin{proof} It is well known that  $\;  r(T) \leq w(T) \leq \Vert T\Vert.$ Thus, for all $X\in \B(H)$, invertible and for each natural number n,
 we have
$$ r(T) = r(\Delta_{\lambda}^n(XTX^{-1})) \leq w(\Delta_{\lambda}^n(XTX^{-1})) \leq \Vert \Delta_{\lambda}^n(XTX^{-1})\Vert$$
It follows that
 \begin{eqnarray*}
 r(T)&  \leq  & \inf\{ w( \Delta_{\lambda}^n(X T X^{-1})); \; \;  X \in \B(H) \;\;\text{invertible}\, \}\\
  & \leq &  \inf\{ w( \Delta_{\lambda}^n(\exp(A) T \exp(-A))); \; \;  A  \in \B(H)  \;\;\text{self adjoint} \,\},\\
   & \leq &  \inf\{ \Vert \Delta_{\lambda}^n(\exp(A) T \exp(-A))\Vert; \; \;  A  \in \B(H)  \;\;\text{self adjoint} \,\}\\
   & = & r(T)  \quad \quad \text{( by Corollary 2.1)}.
 \end{eqnarray*}
Hence we obtain  the desired equalities.
\end{proof}

For a bounded linear operator $S$, we will write $\mathcal{R}e(S) = \frac{1}{2}(S + S^*)$, the real part of $S$. And we denote  by $ \overline{W}(S)$  the closure 
 of the numerical range of $S$. Then we have de following result

\begin{thm}
 For every operator  $T\in\B(H)$,  there exists $\theta \in \mathbb{R}$ such that for all natural number $n$, 
\begin{eqnarray*}
r(T)&=&\inf\{w(\mathcal{R}e(\Delta_{\lambda}^n(\exp(i\theta) XTX^{-1}))), \: X\in\B(H) \;\;\text{invertible}\;\}
\\&=&\inf\{\| \mathcal{R}e(\Delta_{\lambda}^n(\exp(i\theta)XTX^{-1}))\|, \: X\in\B(H) \;\;\text{invertible}\;\}.
\end{eqnarray*}
\end{thm}

\begin{proof} First assume that $r(T) \in \sigma(T)$. Then   for all invertible $X\in \B(H)$, we have
\begin{eqnarray*}
 r(T) \in  \mathcal{R}e(\sigma(T)) = \mathcal{R}e(\sigma(\Delta_{\lambda}^n( XTX^{-1}))).
\end{eqnarray*}
Thus
\begin{eqnarray*}
 r(T) \in  \mathcal{R}e(\sigma(\Delta_{\lambda}^n( XTX^{-1})))
 \subseteq  \mathcal{R}e(\overline{W}(\Delta_{\lambda}^n( XTX^{-1})))
  =  \overline{W}(\mathcal{R}e(\Delta_{\lambda}^n( XTX^{-1})))
\end{eqnarray*}
which implies
\begin{eqnarray*}
 r(T) &\leq&   w( \mathcal{R}e(\Delta_{\lambda}^n(XTX^{-1})))\\
   &\leq& \Vert \mathcal{R}e(\Delta_{\lambda}^n(XTX^{-1}))\Vert\\
    &\leq& \Vert \Delta_{\lambda}^n(XTX^{-1})\Vert.
 \end{eqnarray*}
Since the last inequalities are satisfied for all $X\in \B(H)$ invertible, we obtain
\begin{eqnarray*}
 r(T) &\leq& \inf\{  w( \mathcal{R}e(\Delta_{\lambda}^n(XTX^{-1}))) \: \;X\in\B(H) \;\;\text{invertible}\;\} \\
   &\leq&\inf \{ \Vert \mathcal{R}e(\Delta_{\lambda}^n(XTX^{-1}))\Vert \: \;X\in\B(H) \;\;\text{invertible}\;\}\\
    &\leq& \inf \{ \Vert \Delta_{\lambda}^n(XTX^{-1})\Vert \: \; X\in\B(H) \;\;\text{invertible}\;\}\\
    & = & r(T) \quad \quad \text{(by Corollary 2.1)}.
 \end{eqnarray*}
We have shown that if  $r(T) \in \sigma(T)$ then 
\begin{eqnarray*}
 r(T) &=& \inf\{  w( \mathcal{R}e(\Delta_{\lambda}^n(XTX^{-1}))) \: \;X\in\B(H) \;\;\text{invertible}\;\} \\
   &=&\inf \{ \Vert \mathcal{R}e(\Delta_{\lambda}^n(XTX^{-1}))\Vert \: \;X\in\B(H) \;\;\text{invertible}\;\}\\
    \end{eqnarray*}
Now, if $T$ is an arbitrary operator, then  there exists $z\in \sigma(T)$ such that, $\vert z\vert = r(T)$. Put $\theta = -\arg(z)$.  Then $r(T) = z \exp(i \theta) \in \sigma(\exp(i \theta)T)$.  Hence by the first part of de proof, we conclude that
\begin{eqnarray*}
r(T) = r(  \exp(i \theta)T) &=& \inf\{w(\mathcal{R}e(\Delta_{\lambda}^n(\exp(i\theta) XTX^{-1}))), \: X\in\B(H) \;\;\text{invertible}\;\}
\\&=&\inf\{\| \mathcal{R}e(\Delta_{\lambda}^n(\exp(i\theta)XTX^{-1}))\|, \: X\in\B(H) \;\;\text{invertible}\;\}.
\end{eqnarray*}
This completes the proof of the theorem.
\end{proof}

As immediate consequence of Theorem 3.2, we obtain the following corollary
  which is a characterization of normaloid operators
  \begin{cor} If $T\in\B(H)$, and a natural number $n$, then the following assertions are equivalent

(i) $T$ is normaloid;

(ii) there exists $\theta \in \mathbb{R}$ such that, $\: \text{ for all }\: X\in\B(H) \;\;\text{invertible}$
 $$\Vert T\Vert \leq w(\mathcal{R}e(\Delta_{\lambda}^n(\exp(i\theta) XTX^{-1})));$$

(iii) there exists $\theta \in \mathbb{R}$ such that, $\: \text{ for all }\: X\in\B(H) \;\;\text{invertible}$
 $$\Vert T\Vert \leq \Vert \mathcal{R}e(\Delta_{\lambda}^n(\exp(i\theta) XTX^{-1}))\Vert.$$
\end{cor}
  
 \bigskip

  We end this paper by the following theorem which gives a new formula of the spectral radius of $T$,  in terms of the asymptotic behavior of  powers  and the numerical radius of $T$
\begin{thm} For every operator  $T\in\B(H)$ and for each natural number n,  we have 
\begin{eqnarray*}
r(T) = \lim_k w(\Delta_\lambda^n(T^k))^{1/k}.
\end{eqnarray*}
\end{thm}

\begin{proof}  For each natural number n and  k,  we have 
\begin{eqnarray*}
r(T)^k=r(T^k)  = r(\Delta_\lambda^n(T^k)) \leq  w(\Delta_\lambda^n(T^k))\leq \|\Delta_\lambda^n(T^k)\|
\end{eqnarray*}
Hence 
\begin{eqnarray*}
r(T)  \leq  w(\Delta_\lambda^n(T^k))^{1/k}\leq \|\Delta_\lambda^n(T^k)\|^{1/k}.
\end{eqnarray*}
By Theorem 2.2, we deduce that
\begin{eqnarray*}
r(T) = \lim_k w(\Delta_\lambda^n(T^k))^{1/k}.
\end{eqnarray*}
Which completes the proof.
\end{proof}

\end{document}